\documentclass[reqno]{amsart}
\usepackage{times}
\usepackage{amssymb}
\usepackage{amsmath}
\allowdisplaybreaks[4]

\usepackage{amsmath}
\usepackage{amssymb}
\usepackage{indentfirst}
\usepackage{galois}
\usepackage{amsthm}
\usepackage{amsfonts}

\newtheorem{theorem}{Theorem}[section]
\newtheorem{proposition}[theorem]{Proposition}
\newtheorem{lemma}[theorem]{Lemma}
\newtheorem{corollary}[theorem]{Corollary}

\newtheorem*{TheoremB}{Theorem B}
\newtheorem*{TheoremA}{Theorem A}
\newtheorem*{conjecture}{Conjecture}

\theoremstyle{definition}
\newtheorem{definition}[theorem]{Definition}
\newtheorem{example}[theorem]{Example}

\theoremstyle{remark}
\newtheorem{remark}[theorem]{Remark}

\numberwithin{equation}{section}

\begin{document}

\title[Nevanlinna counting functions and pull-back measures]{Nevanlinna counting functions and pull-back measures on maximal ideal space of $H^\infty$}

\author[Y.X. Gao, Y. Liang, and Z.H. Zhou] {Yong-Xin Gao, Yuxia Liang, and Ze-Hua Zhou*}

\address{\newline  Yong-Xin Gao\newline  School of Mathematical Sciences and LPMC, Nankai University, Tianjin 300071, P.R. China.}
\email{ygao@nankai.edu.cn}

\address{\newline  Yuxia Liang\newline  School of Mathematical Sciences, Tianjin Normal University, Tianjin 300387, P.R. China.}
\email{liangyx1986@126.com}

\address{\newline Ze-Hua Zhou\newline School of Mathematics, Tianjin University, Tianjin 300354, P.R. China.}
\email{zehuazhoumath@aliyun.com; zhzhou@tju.edu.cn}


\date{}
\thanks{\noindent $^*$Corresponding author.\\
This work is supported in part by the National Natural Science
Foundation of China (Grant Nos. 12001293, 11701422, 12171353).}

\begin{abstract}
 In this paper we give precise characterizations of the relation between the Nevanlinna counting function and pull-back measure of an analytic self-map of the unit disk near the boundary. We show that it is quite worth considering these two concepts on the maximal ideal space of the bounded analytic functions.
\end{abstract} \maketitle

\section{introduction}

For a given analytic self-map $\varphi$ of the unit disk $D$, the Nevanlinna counting function $N_\varphi$ is defined as $$N_\varphi(a)=\sum_{\varphi(z)=a}\log\frac{1}{|z|}$$
for $a\in\varphi(D)$. If $a\in D\backslash\varphi(D)$ just define $N_\varphi(a)=0$. This function has its origin in value distribution theory, and plays an important role in the study of the composition operators, see \cite{Cowen} for example.

Another concept this paper concerns with is the pull-back measure $\mu_\varphi$ of an analytic self-map $\varphi$.  It can be defined as a Borel measure on $\mathfrak{M}_{\infty}$. Here $\mathfrak{M}_{\infty}$ denotes the maximal ideal space of the Banach algebra of bounded analytic functions on $D$. The definition of $\mu_\varphi$ is given in Section 3.

There is a strong connection between the Nevanlinna counting function and pull-back measure of $\varphi$. As a typical example, let's recall some results on the properties of the composition operator $C_\varphi$ acting on Hardy space $H^2$. All of these properties can be characterized in terms of either $N_\varphi$ or $\mu_\varphi$.

(a). It is well known that $C_\varphi$ is always bounded on $H^2$. So we can find a constant $c>0$ such that
\begin{align}\label{littlewood}
\int_0^{2\pi}|f\comp\varphi(r\text{e}^{i\theta})|^2d\theta\leqslant c\int_0^{2\pi}|f(r\text{e}^{i\theta})|^2d\theta
\end{align}
for any $f\in H^2$. The famous Carleson embedding theorem, see \cite{Car}, states that (\ref{littlewood}) is equivalent to the following: there exists a constant $c'>0$ such that
$$\mu_\varphi(\hat{W}(\zeta,h))\leqslant c'\cdot h$$
for any $|\zeta|=1$ and $h>0$, where $\hat{W}(\zeta,h)$ denotes the closure of the Carleson window $$W(\zeta,h)=\{r\text{e}^{i\theta}\in D : 1-h<r<1 \text\quad\text{and}\quad|\text{e}^{i\theta}-\zeta|<h\}$$
in $\mathfrak{M}_{\infty}$. From another viewpoint, according to the Littelwood-Paley identity:
$$\|f\comp\varphi\|_2^2=|f(\varphi(0))|^2+2\int_D|f'(w)|^2N_\varphi(w)dv(w),$$
the boundedness of  $C_\varphi$ can also be derived from the fact that the function $\frac{N_\varphi(a)}{-\log|a|}$ is bounded as $|a|\to1$.

(b). It is proved in \cite{Nina} that the composition operator $C_\varphi$ on $H^2$ has closed range if and only if we can find constants $\delta>0$ and $C>0$ such that
$$v(T_\delta\cap W(\zeta,h))\geqslant C\cdot h^2$$
for all $\zeta\in\partial D$ and $h>0$, where $T_\delta=\{z\in D : \frac{N_\varphi(z)}{-\log|z|}>\delta\}$ and $v$ is the normalized Lebesgue measure on $D$.
Whereas a much earlier result in \cite{CLRA} shows that $C_\varphi$ has closed range if and only if there is a constant $C'>0$ such that
$$\mu_\varphi(\hat{F})\geqslant C'\cdot \frac{d\theta}{2\pi}(F)$$
for any Borel subset $F$ of $\partial D$. Here the meaning of $\hat{F}$ is illustrated in our Section 2.

(c). The compactness of a composition operator can also be characterized by using either the Nevanlinna counting function or pull-back measure. In \cite{Sho} Shapiro proves that $C_\varphi$ is compact on $H^2$ if and only if
$$\lim_{|a|\to 1}\frac{N_\varphi(a)}{-\log|a|}=0.$$
While in \cite{Mac} MacCluer shows that the compactness of $C_\varphi$ is also equivalent to the vanishing Carleson measure condition of $\mu_\varphi$, that is,
$$\mu_\varphi(\hat{W}(\zeta,h))=\text{o}(h)$$
as $h\to 0$ uniformly in $\zeta\in\partial D$.

Based on the above, it would be quite interesting to investigate the connection between the Nevanlinna counting function $N_\varphi$ and pull-back measure $\mu_\varphi$ directly without the composition operator $C_\varphi$ as a bridge.

In 2011, Lef\`{e}vre etc \cite{MathA} revealed some relations between these two concepts near the boundary of $D$. They proved that the maximal Nevanlinna counting function and pull-back measure can be dominated by each other up to constants near the boundary. Precisely speaking, one has the following Theorem A, which is the combination of Theorems 3.1 and 4.1 in \cite{MathA}.

\begin{TheoremA}
Suppose $\varphi$ is an analytic self-map of $D$. Then for $\zeta\in\partial D$ and $0<h<(1-|\varphi(0)|)/4$ one has
$$\frac{1}{196}\sup_{W(\zeta,h/24)}N_\varphi\leqslant\mu_\varphi(\hat{W}(\zeta,h))\leqslant64\sup_{W(\zeta,64h)}N_\varphi.$$
\end{TheoremA}
Later in 2016, \cite{new} gave another proof of this theorem.

With a little more efforts, \cite{MathA} also got an improvement of Theorem A: the pull-back measure of a Carleson window can actually be controlled by the average of $N_\varphi$ on a larger Carleson window. Let $S(\zeta,h)=\{z\in D : |z-\zeta|<h\},$ and $\hat{S}(\zeta,h)$ be the closure of $S(\zeta,h)$ in $\mathfrak{M}_{\infty}$. Recall that $S(\zeta,h)$ are equivalent choices of the Carleson windows $W(\zeta,h)$. The following Theoren B is Theorem 4.2 in \cite{MathA}:

\begin{TheoremB} Suppose $\varphi$ is an analytic self-map of $D$. Then for $\zeta\in\partial D$ and $0<h<(1-|\varphi(0)|)/8$ one has
$$\mu_\varphi(\hat{S}(\zeta,h))\leqslant\frac{128\times17^2}{v(S(\zeta,17h))}\int_{S(\zeta,17h)}N_\varphi(z) dv(z).$$
\end{TheoremB}

In this paper, we will give a precise equality relating this two concepts on Carleson windows. As corollaries, we can get Theorem A and Theorem B easily, with much better constants.

However, we are not content with these results. From the composition operator's properties (a) - (c) above, we can see that it should be much more important to relate the pull-back measure of $\varphi$ to the function $\frac{N_\varphi(z)}{-\log|z|}$ other than $N_\varphi(z)$. This is another main aim of this paper. We will prove that the average of $\frac{N_\varphi(z)}{-\log|z|}$ on a Carleson window equals to the average of the pull-back measures of some other Carleson windows. As a consequence, we can show that the pull-back measure and the function $\frac{N_\varphi(z)}{-\log|z|}$ can be dominated by each other up to constants.

This paper is organized as follows. In Section 3, we will introduce the definition of the pull-back measure $\mu_\varphi$ on $\mathfrak{M}_{\infty}$, and prove some properties of it. Then in Section 4 we study the relationship between $N_\varphi$ and $\mu_\varphi$ near the boundary. We will give a precise version of Theorem B in this section, which is Theorem \ref{coro1}. Finally in Section 5 we will investigate the connection between $\frac{N_\varphi(z)}{-\log|z|}$ and $\mu_\varphi$. Then we will present some corollaries and applications. The main result of this section is Theorem \ref{aver2}.

\section{preliminaries}

First of all, let's introduce some notations used throughout this paper.

Let $H^\infty$ be the space of bounded analytic functions on the unit disk $D$, and $L^\infty(\partial D)$ be the space of $L^\infty$ functions on the unit circle $\partial D$. Both of the spaces are equipped with the supremum norm $\|\cdot\|_\infty$, making them to be Banach algebras.

The \textit{maximal ideal space} of a Banach algebra is the collection of all non-zero multiplicative linear functionals on it, equipping with the weak * topology. We use $\mathfrak{M}_{\infty}$ to denote the maximal ideal space of $H^\infty$. By Corona Theorem, $\mathfrak{M}_{\infty}$ is a compact Hausdorff space that contains $D$ as a dense subset.

The maximal ideal space of $L^\infty(\partial D)$ will be denoted by $\mathfrak{M}_{L^{\infty}}$. It is well known that $\mathfrak{M}_{L^{\infty}}$ is exactly the Shilov boundary of $\mathfrak{M}_{\infty}$, see Theorem 1.7 in Chapter \uppercase\expandafter{\romannumeral5} of \cite{BAF} for example.

The \textit{Gelfand transform} of a function $f\in H^\infty$, denoted by $\hat{f}$, is defined as
$$\hat{f}(m)=m(f) \quad, \quad m\in\mathfrak{M}_\infty.$$
It is a continuous function on $\mathfrak{M}_{\infty}$ and coincides with $f$ on $D$. Therefore, by identifying a function $f$ with its Gelfand transform $\hat{f}$, we can regard $H^\infty$ as a subalgebra of $C(\mathfrak{M}_\infty)$. The Gelfand transform on $L^\infty(\partial D)$ is defined in a similar way.

For each $m\in\mathfrak{M}_\infty$ there exists a unique positive Borel measure $\sigma_m$ on $\mathfrak{M}_{L^\infty}$ such that $\sigma_m(\mathfrak{M}_{L^\infty})=1$ and
$$\hat{f}(m)=\int_{\mathfrak{M}_{L^\infty}}\hat{f}d\sigma_m$$
for all $f\in H^\infty$. $\sigma_m$ is called the \textit{representing measure for $m\in\mathfrak{M}_\infty$}. In the rest part of this paper, we will denote the representing measure for the point $0$ briefly by $\sigma$. By the Riesz representation theorem, $\sigma$ is just the measure determined by the linear functional $$g\mapsto\int_0^{2\pi}g(\text{e}^{i\theta})\frac{d\theta}{2\pi}$$ for $g\in L^\infty(\partial D)=C(\mathfrak{M}_{L^\infty})$.

For any Borel subset $F \subset \partial D$, let $\chi_F \in L^{\infty}(\partial D)$ be the characteristic function of $F$. Then $\hat{\chi}_F$ is a continuous fuction on $\mathfrak{M}_{L^{\infty}}$ that assumes only the value $0$ and $1$. Define $$\hat{F}=\{m\in\mathfrak{M}_{L^{\infty}} : \hat{\chi}_F(m)=1\},$$ then it is easy to see that $\hat{F}$ is a clopen subset in $\mathfrak{M}_{L^{\infty}}$, and $\sigma(\hat{F})=\frac{d\theta}{2\pi}(F)$.

\begin{remark}
Although $\hat{\chi}_F$ is defined on $\mathfrak{M}_{L^{\infty}}$, in what follows we will always identify $\hat{\chi}_F$ with its Poisson integral, which is a function on $\mathfrak{M}_{\infty}$. That is,
$$\hat{\chi}_F(m)=\int_{\mathfrak{M}_{L^{\infty}}}\hat{\chi}_Fd\sigma_m$$
for each $m\in\mathfrak{M}_{\infty}$. It is easy to see that $\hat{\chi}_F$ is continuous on $\mathfrak{M}_{\infty}$ and harmonic on $D$. In fact, by Theorem 4.5 in Chapter \uppercase\expandafter{\romannumeral2} of \cite{BAF}, there exists a function $h$ invertible in $H^\infty$ such that $\chi_F=\log|h|$ almost everywhere on $\partial D$. Then on $\mathfrak{M}_{\infty}$ we have $\hat{\chi}_F$ equals to $\log|\hat{h}|$.
\end{remark}

For each $\alpha\in(0,1)$, we set $G_F (\alpha) = \{m \in \mathfrak{M}_{\infty} : \hat{\chi}_F (m) \geqslant\alpha\}$, and $L_F(\alpha)=\{z\in D : \hat{\chi}_F(z)=\alpha\}$.

\begin{lemma}\label{geo}
Suppose F is an arc on $\partial D$, then $z\in G_F (\alpha)\cap D$ if and only if $z$ lies in the disc (or half plane) intersect with $\partial D$ in $F$, and at an angle of $\pi (1-\alpha)$.
\end{lemma}

\begin{proof}
Without loss of generality, we may assume that $F=\{\text{e}^{i\theta} : \theta_0<\theta<\pi-\theta_0\},$ where $\theta_0\in (-\frac{\pi}{2},\frac{\pi}{2})$. Then for $z\in D$,

\begin{align*}
 \chi_F(z)=&\text{Re}\int_{\theta_0}^{\pi-\theta_0}\frac{\text{e}^{i\theta}+z}{\text{e}^{i\theta}-z}\cdot\frac{d\theta}{2\pi}\nonumber\\
 =&\frac{1}{\pi}\arg\frac{-\text{e}^{-i\theta_0}-z}{\text{e}^{i\theta_0}-z}+\frac{\theta_0}{\pi}-\frac{1}{2}.
\end{align*}
So $\chi_F (z)\geqslant\alpha$ if and only if $$\arg\frac{-\text{e}^{-i\theta_0}-z}{\text{e}^{i\theta_0}-z}\geqslant\pi \alpha+\frac{\pi}{2}-\theta_0.$$ This means that $z$ lies in the open disc intersect with $\partial D$ in $F$, and at an angle $\pi (1-\alpha)$

\end{proof}

\begin{remark}
Lemma \ref{geo} shows that for any fixed $\alpha \in(0,1)$, the sets $G_F(\alpha)$ are equivalent choices of the classical Carleson windows $W(\zeta,h)$ or $S(\zeta,h)$. In fact, there exist positive constants $C$ and $C'$, depending only on $\alpha$, such that for each arc $F$ with centre $\zeta$ and length $2h$ on $\partial D$ we have
$$\hat{W}(\zeta,Ch)\subset G_F(\alpha)\subset\hat{W}(\zeta,C'h).$$
  Note that sets $G_F(\alpha)$ are M\"obius invariant, i.e., for any automorphism $\phi$ of $D$, since $\hat{\chi}_{\phi(F)}=\hat{\chi}_F\comp\phi^{-1}$, therefore $$\phi(G_F(\alpha))=G_{\phi(F)}(\alpha).$$ The advantage of using $G_F(\alpha)$ instead of $W(\zeta,h)$ or $S(\zeta,h)$ can be shown in the proof of Theorem \ref{aver2}.
\end{remark}

\section{carleson measure on maximal ideal space of $H^{\infty}$}

Each analytic self-map $\varphi$ of $D$ has a unique continuous extension as a self-map of $\mathfrak{M}_{\infty}$, that is,
$$\varphi(m)(f)=m(f\comp\varphi)$$
for $m\in\mathfrak{M}_{\infty}$ and $f\in H^\infty$. In fact, we can prove that $\varphi$ is analytic on each Gleason part in the sense of the following Proposition \ref{ana}. Recall that the \textit{pseudo-hyperbolic distance on }$\mathfrak{M}_{\infty}$ is defined as $$\rho(m_1,m_2)=\{|\hat{f}(m_2)| : f\in H^\infty, \|f\|_\infty\leqslant1, \hat{f}(m_1)=0\}$$
for $m_1, m_2\in\mathfrak{M}_{\infty}$. The set
$$P(m)=\{m'\in\mathfrak{M}_{\infty} : \rho(m',m)<1\}$$
is called \textit{the Gleason part containing $m$}. For each $a\in D$, let $L_a(z)=\frac{a+z}{1+\overline{a}z}$.
It was proved by Hoffman in \cite{Hoff} that if $\{z_\alpha\}$ is a net in $D$ converging to a point $m\in\mathfrak{M}_{\infty}$, then $L_{z_\alpha}$ converges to a continuous map, denoted by $L_m$, in the product topology in $\left(\mathfrak{M}_\infty\right)^D$. It was also shown in \cite{Hoff} that $L_m$ maps the unit disk $D$ onto the Gleason part $P(m)$, and $L_m$ is either a constant or an injection. As a consequence, each Gleason part is either a single point or an analytic disk.
\begin{proposition}\label{ana}
Suppose $\varphi$ is an analytic self-map of $D$. Then for each $m\in\mathfrak{M}_{\infty}$, there exists an analytic self-map $\phi$ of $D$ with $\phi(0)=0$ such that
$$\varphi\comp L_m(z)=L_{\varphi(m)}\comp\phi(z)$$
for $z\in D$.
\end{proposition}

\begin{proof}
Fix $m\in\mathfrak{M}_{\infty}$. Suppose $\{z_\alpha\}$ is a net in $D$ that converges to $m$. Let
$$\phi_\alpha=L_{\varphi(z_\alpha)}^{-1}\comp\varphi\comp L_{z_\alpha}.$$
Then $\phi_\alpha$ is analytic on $D$ and $\phi_\alpha(0)=0$. Since $\{\phi_\alpha\}\subset\overline{D}^D$ which is compact, by passing to a subnet we may assume that $\{\phi_\alpha\}$ converges to a $\phi\in\overline{D}^D$. By a normal family argument we can see this convergence is uniform on compact subsets in $D.$ Note $\phi(0)=0$, so $\phi$ is an analytic self-map of $D$. Since
$$\varphi\comp L_{z_\alpha}=L_{\varphi(z_\alpha)}\comp\phi_\alpha.$$
Thus the final result follows by lemma 2.2 in \cite{Budde},
\end{proof}

\begin{corollary}
Suppose $\varphi$ is an analytic self-map of $D$. Then
$$\rho(\varphi(m_1),\varphi(m_2))\leqslant\rho(m_1,m_2)$$
for any $m_1,m_2\in\mathfrak{M}_{\infty}$. Particularly, for any $m\in\mathfrak{M}_{\infty}$, the set $\varphi(P(m))$ is contained in $P(\varphi(m))$.
\end{corollary}

\begin{proof}
If $\rho(\varphi(m_1),\varphi(m_2))$ is $0$ or $1$, the result is obvious.

Now assume that $0<\rho(m_1,m_2)<1$, then $m_2$ and $m_1$ belong to a same nontrivial Gleason part. So there exists $w\in D$ such that $m_2=L_{m_1}(w)$. Thus by Proposition \ref{ana}, there exists an analytic self-map $\phi$ of $D$ such that $\phi(0)=0$ and
\begin{align*}
\varphi(m_2)=L_{\varphi(m_1)}\comp\phi(w).
\end{align*}
Also note that $\varphi(m_1)=L_{\varphi(m_1)}(0)$. So by (6.12) in \cite{Hoff},
\begin{align*}
\rho(\varphi(m_1),\varphi(m_2))\leqslant\rho(0,\phi(w))\leqslant&\rho(0,w)
=\rho(m_1,m_2).
\end{align*}
\end{proof}

A result of M. Behrens states that $\varphi$ is an inner function if and only if $\varphi$ maps $\mathfrak{M}_{L^\infty}$ into itself. We now prove a stronger result here.

\begin{theorem}\label{Bud}
Suppose $\varphi$ is an analytic self-map of $D$. Let $F=\{\zeta\in\partial D : |\varphi^*(\zeta)|=1\}$, where $\varphi^*(\zeta)$ represents the radial limit of $\varphi$ at $\zeta$. Then $\varphi$ maps $\hat{F}$ into the Shilov boundary $\mathfrak{M}_{L^\infty}$.
\end{theorem}

\begin{proof}
Fix an arbitrary inner function $u$. Then by Proposition 2.25 in \cite{Cowen}, $(u\comp\varphi)^*=u^*\comp\varphi^*$ almost everywhere on $\partial D$, so $|(u\comp\varphi)^*(\zeta)|=1$ for almost every $\zeta\in F$. Therefore we have
$|u\comp\varphi|\cdot\chi_F$ equals to $\chi_F$ as functions in $L^\infty(\partial D)$, hence the Gelfand transforms of $\chi_F\cdot |u\comp\varphi|$ and $\chi_F$ are the same.  So for $m\in\hat{F}$, one has
$$\hat{\chi}_F(m)\cdot |\hat{u}\comp\varphi(m)|=\hat{\chi}_F(m)=1,$$
i.e., $|\hat{u}(\varphi(m))|=1$. By a theorem of Newman, see Theorem 2.2 in Chapter \uppercase\expandafter{\romannumeral5} of \cite{BAF}, we must have $\varphi(m)$ belongs to $\mathfrak{M}_{L^\infty}$.
\end{proof}

%\begin{theorem}
%Suppose $\mu$ is a finite positive Borel measure on $D$. Then $\mu$ is a Carleson measure in the sense of Definition *** if and only if $\mu$ satisfies the classical Carleson's condition, that is, there exist $K>0$ such that $\mu\left(W(\zeta,h)\right)<Kh$ for $\zeta\in \partial D$ and $0<h<1$.
%\end{theorem}

The \textit{pull-back measure} $\mu_\varphi$ of an analytic self-map $\varphi$ is defined as
$$\mu_\varphi(E)=\sigma\left(\varphi^{-1}(E)\cap\mathfrak{M}_{L^\infty}\right)$$
for any Borel subset $E$ of $\mathfrak{M}_{\infty}$. It is easy to check that $\mu_\varphi$ is always a Carleson measure on $\mathfrak{M}_{\infty}$. In fact, according to the proof of the boundness of $C_\varphi$ on $H^2$, see the proof of Theorem 5.1.5 in \cite{237} for example, we can know that for any Borel set $F\subset \partial D,$
$$\int_{\mathfrak{M}_{\infty}}\hat{\chi}_Fd\mu_\varphi\leqslant C\int_{\mathfrak{M}_{L^{\infty}}}\hat{\chi}_Fd\sigma=C\sigma(\hat{F}),$$
where $C=\frac{1+|\varphi(0)|}{1-|\varphi(0)|}$. So
\begin{align*}
  \mu_\varphi\left(G_F(\alpha)\right)\leqslant\frac{C}{\alpha}\int_{G_F(\alpha)}\hat{\chi}_Fd\mu_\varphi\leqslant
  \frac{C}{\alpha}\cdot\sigma(\hat{F}).
\end{align*}

\begin{proposition}\label{pms}
Suppose $\varphi$ is an analytic self-map of $D$. Then
$\mu_\varphi(E)=0$
whenever $E$ is contained in $\mathfrak{M}_{\infty}\backslash(D\cup\mathfrak{M}_{L^\infty})$.
\end{proposition}

\begin{proof}
Let $F_n=\{\zeta\in\partial D : |\varphi^*(\zeta)|<1-1/n\}$, and let $g_n=\varphi^*\cdot\hat{\chi}_{F_n}\in L^\infty(\partial D)$. Then $\|g_n\|_\infty<1$. So for any $m\in\hat{F}_n$, we have $|\hat{\varphi}(m)|=|\hat{g}_n(m)|<1$. This means that $\varphi(m)\in D$. Therefore we have $\varphi(\hat{F}_n)$ is contained in $D$.

Now let $F_0=\{\zeta\in\partial D : |\varphi^*(\zeta)|=1\}$. By Theorem \ref{Bud}, $\varphi(\hat{F}_0)\subset\mathfrak{M}_{L^\infty}$. Note that $\partial D\backslash \bigcup_{n=0}^\infty F_n$ is of measure zero, so
\begin{align*}
\sigma(\bigcup_{n=0}^\infty\hat{F}_n)=&\sigma(\bigcup_{n=1}^\infty\hat{F}_n)+\sigma(\hat{F}_0)=1.
\end{align*}
Therefore we have
\begin{align*}
\mu_\varphi\left(\mathfrak{M}_{\infty}\backslash(D\cup\mathfrak{M}_{L^\infty})\right)\leqslant\sigma(\mathfrak{M}_{L^\infty}\backslash\bigcup_{n=0}^\infty\hat{F}_n)=0.
\end{align*}
\end{proof}

\section{nevanlinna counting functions near the boundary}

%\begin{align*}
 % N_{\varphi} (w) = - \log |w - \varphi (0) | + \int_{\partial D} \log |w -
  %\varphi^{\ast} (\zeta) | d \sigma (\zeta) - \mu_a  (\partial D)
%\end{align*}

%for $w \ne \varphi (0)$, where $\mu_a$ is the singular measure associate with $\frac{a-\varphi(z)}{1-\overline{a}\varphi(z)}.$ It is well known that $\mu_a(\partial D)=0$ for all $a\in D$ except for a set of logarithmic capacity zero.

%Continuing to use the notation in {\cite{Kim2}}, we set
%\begin{align*}
 % \overline{N_{\varphi}} (w) = N_{\varphi} (w) + \mu_a  (\partial D) = - \log
  %|w - \varphi (0) | + \int_{\partial D} \log |w - \varphi^{\ast} (\zeta) | d
  %\sigma (\zeta) .
%\end{align*}

%Suppose $\varphi (0) = 0$, then
%\[ \frac{\partial \overline{N_{\varphi}} (\rho \text{e}^{i \theta})}{\partial
 % \rho} = - \frac{1}{2 \rho} + \frac{1}{2 \rho}  \int_{\partial D}
  %P_{\frac{\varphi^{\ast} (\zeta)}{\rho}} (\text{e}^{i \theta}) d \sigma
   %(\zeta) . \]
%Note that $P_a (\text{e}^{i \theta}) = - P_{\frac{1}{\bar{a}}} (\text{e}^{- i
%\theta})$,

%For any Borel set $F\subset \partial D$,
%\begin{align}
%\int_F\frac{\partial \overline{N_{\varphi}} (\rho \text{e}^{i \theta})}{\partial\rho}\frac{d\theta}{2\pi}&=-\frac{\sigma(F)}{2\rho}+\frac{1}{2 \rho}\int_F\frac{d\theta}{2\pi}  \int_{\partial D}P_{\frac{\varphi^{\ast} (\zeta)}{\rho}} (\text{e}^{i \theta}) d \sigma(\zeta)\nonumber\\
%&=-\frac{1}{2 \rho}\chi_F(0)+\frac{1}{2 \rho}\int_{\partial D}\chi_F(\frac{\varphi^{\ast} (\zeta)}{\rho})d\sigma(\zeta)
%\end{align}

The aim of this section is to give a precise relation between the Nevanlinna counting function and pull-back measure on Carleson windows. The main result in this section is Theorem \ref{coro1}.

Our main tool is the following Stanton's formula (see \cite{Stan}, Theorem 2): Suppose $\varphi$ is an analytic self-map of $D$ and $G$ is a subharmonic function from $D$ to $\mathbb{R}$. Then
\begin{align}\label{Stan}
\lim_{r\to 1}\int_{\partial D}G(\varphi(r\zeta))\frac{d\theta}{2\pi}(\zeta)=G(\varphi(0))+\frac{1}{2}\int_D\Delta G(z)N_\varphi(z)dv(z).
\end{align}
The following formula, which is Theorem 8.16 in \cite{Rudin}, will also be used repeatedly in this paper: Suppose $f$ : $X\to[0,+\infty)$ is Borel measurable on $(X,d\mu)$. Function $H$ is monotonic on $[0,+\infty)$, absolutely continuous on $[0,T]$ for all $T>0$, and $H(0)=0$. Then
\begin{align}\label{Rudin}
\int_XH\comp f d\mu=\int_0^\infty\mu\{f>t\}H'(t)dt.
\end{align}
The derivatives in (\ref{Stan}) and (\ref{Rudin}) are taken as distributions.

It is well known that the function $N_\varphi$ satisfies the sub-mean value property, see Lemma 3.18 in \cite{Cowen}. Futhermore, it is shown in \cite{complete} that $N_\varphi$ equals to the function
\begin{align*}
 \overline{N}_{\varphi} (a)=-\log
  |a - \varphi (0) | + \int_{\partial D} \log |a - \varphi^{\ast} (\zeta) | d
  \sigma (\zeta)
\end{align*}
outside a set of logarithmic capacity zero in $D$. Note that $\int_{\partial D} \log |a -\varphi^{\ast} (\zeta) | d\sigma (\zeta)$ is the potential of $\mu_\varphi$, so by Theorem 3.1.2 in \cite{pot}, $\overline{N}_{\varphi}$ is subharmonic on $D\backslash \varphi(0)$. (Also see Theorem 2.3 in \cite{complete}.)

First we begin with the following proposition. In what follows, we use $ds$ to denote the Lebesgue line element on the plane.
\begin{proposition}\label{aver1}
Suppose $\varphi$ is an analytic self-map of $D$. For a given bounded linear fractional map $f(z)=\frac{az+b}{cz+d}$ on $D$ with $ad-bc=1$, let $G_f(t) = \{ m \in \mathfrak{M}_{\infty} : | \hat{f} (m) | \geqslant t \}$ and $L_f(t)=\{z\in D : |f(z)|=t\}$. Define
\begin{align*}
J_t=\frac{1}{2\pi t}\int_{L_f(t)}\frac{N_\varphi(z)}{|cz+d|^2}ds(z).
\end{align*}

It follows that

(\romannumeral1) if $\varphi(0)\in G_f(t)$, then
\begin{align*}
J_{t}=&-\log\frac{|f\comp\varphi(0)|}{t}+\int_{t}^\infty\frac{\mu_\varphi(G_f(u))}{u}du,\\
%=&\overline{N_\varphi}(\xi)+\int_0^{t}\frac{\mu_\varphi(\mathfrak{M}_{\infty}\backslash G_f(u))}{u}du
\end{align*}

(\romannumeral2) if $\varphi(0)\notin G_f(t)$, then
\begin{align*}
J_{t}=&\int_{t}^\infty\frac{\mu_\varphi(G_f(u))}{u}du.\\
%=&\overline{N_\varphi}(\xi)+\log\frac{|f\comp\varphi(0)|}{t}+\int_0^{t}\frac{\mu_\varphi(\mathfrak{M}_{\infty}\backslash G_f(u))}{u}du.
\end{align*}
%Here $\xi$ denotes the zero of $f$.

\end{proposition}

\begin{proof}
Take $\Phi(t)=\log t$. Fix $t_2 > t_1 > 0$ and define
$$H (t) = \left\{ \begin{array}{ll}
     \Phi (t_1), & 0 \leqslant t < t_1,\\
     \Phi (t), & t_1 \leqslant t \leqslant t_2,\\
     \Phi (t_2), & t > t_2.
   \end{array} \right.$$
By (\ref{Stan}) we have
\begin{align}\label{zuoyou1}
  \int_{\mathfrak{M}_{L^{\infty}}} H (|\hat{f} \comp \varphi |) d\sigma = H (|f \comp \varphi (0) |) + \frac{1}{2}  \int_D \Delta (H \comp|f|) \cdot
  N_{\varphi} dv.
\end{align}

The left side of (\ref{zuoyou1}) equals to
\begin{align}
  \int_{\mathfrak{M}_{\infty}} H (|\hat{f}|) d \mu_{\varphi} \nonumber= &\Phi (t_1)+ \int_0^\infty H'(u) \cdot\mu_{\varphi}(G_f(u))du\nonumber \\
  = & \Phi (t_1)+ \int_{t_1}^{t_2} \frac{\mu_{\varphi}(G_f(u))}{u}du,
\end{align}
where the first equality follows from (\ref{Rudin}).

Now let's deal with the right hand side of (\ref{zuoyou1}). First note that
\begin{align}
H (|f \comp \varphi (0) |) = \left\{ \begin{array}{ll}
     \Phi (t_1), & \varphi(0)\notin G_f(t_1),\\
     \Phi (|f \comp \varphi (0) |), & \varphi(0)\in G_f(t_1)\backslash G_f(t_2),\\
     \Phi (t_2), & \varphi(0)\in G_f(t_2).
   \end{array} \right.
\end{align}
Let $\tau=1/f$, then $\tau$ maps $G_f(t)\cap D$ onto $\overline{D(0,1/t)}\cap \tau(D)$. Since
\begin{align*}
\Delta(H\comp|f|)=&\left(H''(|f|)+\frac{H'(|f|)}{|f|}\right)\cdot|f'|^2\\
=&\Big(H''(|f|)\cdot|f|^4+H'(|f|)\cdot|f|^3\Big)\cdot|\tau'|^2,
\end{align*}
we have
\begin{align}
 &\frac{1}{2}  \int_D \Delta( H \comp|f|) \cdot N_{\varphi} dv\nonumber\\
 =&\frac{1}{2}  \int_D \Big(H''(|f|)\cdot|f|^4+H'(|f|)\cdot|f|^3\Big) N_{\varphi}\cdot|\tau'|^2  dv\nonumber\\
 =&\frac{1}{2\pi}\int_0^\infty dr\int_{\tau(L_{\frac{1}{r}})} \left( H''(\frac{1}{r})\cdot\frac{1}{r^4} +H' (\frac{1}{r})\cdot\frac{1}{r^3} \right)N_\varphi\comp\tau^{-1}(w)ds(w)\nonumber\\
 =&\int_0^\infty\left(u\cdot H'(u)\right)'\cdot J_udu,\nonumber
\end{align}
where the last equality follows by a change of variable $u=1/r.$ The sub-mean value property of $N_\varphi$ guarantees that $J_u$ is continuous with respect to $u$. And $\left(u\cdot H'(u)\right)'=\delta_{t_1}-\delta_{t_2},$
where $\delta_{t_j}$ represents the Dirca measure on $\mathbb{R}$ at $t_j$, which is a distribution of order zero. So
\begin{align}\label{dd}
 \frac{1}{2}  \int_D \Delta( H \comp|f|) \cdot N_{\varphi} dv=J_{t_1}-J_{t_2}.
\end{align}

Now combining (\ref{zuoyou1}) - (\ref{dd}), if $\varphi(0)\notin G_f(t_1)$ then
$$J_{t_1}-J_{t_2}=\int_{t_1}^{t_2} \frac{\mu_{\varphi}(G_f(u))}{u}du.$$
Note that $\lim_{t_2\to\infty}J_{t_2}=0$. Therefore we have
\begin{align*}
J_{t_1}=\int_{t_1}^\infty\frac{\mu_\varphi(G_f(u))}{u}du.
\end{align*}

If $\varphi(0)\in G_f(t_1)\backslash G_f(t_2)$, then
$$J_{t_1}-J_{t_2}=\Phi(t_1)-\Phi (|f \comp \varphi (0) |)+\int_{t_1}^{t_2} \frac{\mu_{\varphi}(G_f(u))}{u}du.$$
So we have
\begin{align*}
J_{t_1}=-\log\frac{|f\comp\varphi(0)|}{t_1}+\int_{t_1}^\infty\frac{\mu_\varphi(G_f(u))}{u}du.
\end{align*}
\end{proof}

As a simple application, we can calculate the integration of $N_\varphi(z)$ on any circle or disk contained in $D$ in terms of the pull-back measure, which is the main result in \cite{Kim}. Note that the proof in \cite{Kim} is quite different form ours here.

\begin{corollary}
[Theorem 3.1 in \cite{Kim}] Suppose $\varphi$ is an analytic self-map of $D$. Then

(i)$$\int_0^{2\pi}N_\varphi(r\text{e}^{i\theta})\frac{d\theta}{2\pi}=-\log^+\frac{|\varphi(0)|}{r}+
\int_r^1\frac{\mu_\varphi(\mathfrak{M}_{\infty}\backslash D(0,u))}{u}du,$$
where $\log^+x=\max\{0, \log x\}$,

(ii)\begin{align*}
\int_{D(0,r)}N_\varphi(z)dv(z)=r^2\overline{N}_\varphi(0)+r^2\int_{D(0,r)}\left(\frac{|z|^2}{2r^2}-\log\frac{|z|}{r}-\frac{1}{2}\right)d\mu_\varphi(z)
\end{align*}
for $0<r<|\varphi(0)|.$

\end{corollary}
\begin{proof}

Take $f(z)=z$ in Proposition \ref{aver1}, then $G_f(t)=\mathfrak{M}_{\infty}\backslash D(0,t)$ and
\begin{align*}
J_t=\frac{1}{2\pi t}\int_{|z|=t}N_\varphi(z)ds(z).
\end{align*}
Then (\textit{\romannumeral1}) follows directly from Proposition \ref{aver1}.

For (\textit{\romannumeral2}), since $\varphi(0)\in G_f(r)$, so for any $r'\in(0,r)$ we have $\varphi(0)\in G_f(r')$, whence by Proposition \ref{aver1} we have
\begin{align*}
J_r-J_{r'}=&\log\frac{r}{r'}-\int_{r'}^r\frac{\mu_\varphi(G_f(u))}{u}du\\
=&\int_{r'}^r\frac{D(0,u)}{u}du.
\end{align*}
Note that $\lim_{r'\to0}J_{r'}=\overline{N}_\varphi(0)$, so
\begin{align*}
J_r=\overline{N}_\varphi(0)+\int_{0}^r\frac{D(0,u)}{u}du.
\end{align*}
Therefore,
\begin{align*}
\int_{D(a,r)}N_\varphi(z)dv(z)=&2\int_0^rtJ_tdt\\
=&2\int_0^r\left[t\overline{N}_\varphi(0)+t\int_0^t\frac{\mu_\varphi(D(0,u))}{u}du\right]dt\\
=&r^2\overline{N}_\varphi(0)+\int_0^r\frac{r^2-u^2}{u}\mu_\varphi(D(0,u)))du\\
=&r^2\overline{N}_\varphi(0)+r^2\int_{D(0,r)}\left(\frac{|z|^2}{2r^2}-\log\frac{|z|}{r}-\frac{1}{2}\right)d\mu_\varphi,
\end{align*}
where the last equality follows from (\ref{Rudin}).

\end{proof}

We are more interested in the integrations of $N_\varphi$ on Carleson windows, which can also be calculated with the help of Proposition \ref{aver1}. Recall that $S(\zeta,h)=\{z\in D : |z-\zeta|<h\}$. Now we can give our main theorem in this section.

\begin{theorem}\label{coro1}
Suppose $\varphi$ is an analytic self-map of $D$. If $\varphi(0)\notin S(\zeta,h)$, then
\begin{align*}
\int_{S(\zeta,h)}N_\varphi(z)dv(z)=\int_0^h\frac{h^2-u^2}{u}\mu_\varphi(\hat{S}(\zeta,u))du.
\end{align*}
\end{theorem}

\begin{proof}
Take $f(z)=z-\zeta$ in Proposition \ref{aver1}, then $G_f(h)=\mathfrak{M}_{\infty}\backslash\hat{S}(\zeta,h)$. So
$$J_h=\frac{1}{2\pi h}\int_{\{z\in D : |z-\zeta|=h\}}N_\varphi(z)ds(z).$$
Now if we take $h'\in(0,h)$, then by Proposition \ref{aver1} we have
\begin{align*}
J_{h'}-J_h&=\log\frac{h'}{h}+\int_{h'}^h\frac{\mu_\phi(G_f(u))}{u}du\\
&=-\int_{h'}^h\frac{\mu_\phi(\hat{S}(\zeta,u))}{u}du.
\end{align*}
Since $\lim_{|z|\to1}N_\varphi(z)=0$, so $J_{h'}\to 0$ as $h'\to 0$. Therefore we have

$$J_h=\int_0^h\frac{\mu_\varphi\left(\hat{S}(\zeta,u)\right)}{u}du.$$
And then
\begin{align*}
\int_{S(\zeta,h)}N_\varphi(z)dv(z)=&2\int_0^htJ_tdt\\
=&2\int_0^hdt\int_0^t\frac{t\cdot\mu_\varphi(\hat{S}(\zeta,u))}{u}du\\
=&\int_0^h\frac{h^2-u^2}{u}\mu_\varphi(\hat{S}(\zeta,u))du.
\end{align*}
\end{proof}

As a corollary, Theorem B (Theorem 4.2 in \cite{MathA}) can be got easily from our Theorem \ref{coro1}, with much better constants.

\begin{corollary}
Suppose $\varphi$ is an analytic self-map of $D$. If $\varphi(0)\notin S(\zeta,2h)$, then
$$\mu_\varphi(\hat{S}(\zeta,h))<\frac{2}{v(S(\zeta,2h))}\int_{S(\zeta,2h)}N_\varphi(z)dv(z)$$
\end{corollary}

\begin{proof}

It is easy to see that $v(S(\zeta,2h))$ is less than $2h^2$, so according to Theorem \ref{coro1},

\begin{align*}
\frac{1}{v(S(\zeta,2h))}\int_{S(\zeta,2h)}N_\varphi(z)dv(z)&>\frac{1}{2h^2}\int_h^{2h}\frac{4h^2-u^2}{u}\mu_\varphi(\hat{S}(\zeta,u))du\\
&\geqslant\frac{\mu_\varphi(\hat{S}(\zeta,h))}{2h^2}\int_h^{2h}\frac{4h^2-u^2}{u}du\\
&>\frac{\mu_\varphi(\hat{S}(\zeta,h))}{2}.
\end{align*}

\end{proof}

\section{average of $\frac{N_\varphi(z)}{-\log|z|}$ near the boundary}

As we have mentioned in Section 1, the behavior of $\frac{N_\varphi(z)}{-\log|z|}$ near the boundary should be more closely related to the pull-back measure $\mu_\varphi$. Since $\frac{-\log|z|}{1-|z|}\to 1$ as $|z|\to1$, we may use the function $\frac{2N_\varphi(z)}{1-|z|^2}$ instead of $\frac{N_\varphi(z)}{-\log|z|}$ in some cases, for the sake of simplifying the calculations. The main result of this section is Theorem \ref{aver2}.

To investigate the relationship between the function $\frac{N_\varphi(z)}{1-|z|^2}$ and measure $\mu_\varphi$, it is more convenience to use the Carleson windows of the form $G_F(\alpha)$, other than $W(\zeta,h)$ or $S(\zeta,h)$. Recall that for a Borel set $F\subset\partial D$, we set $L_F(\alpha)=\{z\in D : \hat{\chi}_F(z)=\alpha\}$ and $G_F(\alpha)=\{m\in\mathfrak{M}_{\infty} : \hat{\chi}_F(z)\geqslant\alpha\}$.

Before we can prove Theorem \ref{aver2}, we need the following two lemmas.

\begin{lemma}\label{lemma1}
Let $F$ be an arc on $\partial D$. Then for any $\alpha\in(0,1)$, on $L_F(\alpha)$ one has
$$\left|\frac{\partial\hat{\chi}_F}{\partial z}\right|=\frac{\sin\pi\alpha}{\pi(1-|z|^2)}.$$
\end{lemma}

\begin{proof}
Without loss of generality, we may assume that $F=\{\text{e}^{i\theta} : \theta_0<\theta<\pi-\theta_0\}$ where $\theta\in(-\frac{\pi}{2},\frac{\pi}{2})$. Then
$$\hat{\chi}_F(z)=\int_{\theta_0}^{\pi-\theta_0}P_{z}(\text{e}^{i\theta})\frac{d\theta}{2\pi},$$
where the Possion kernel $P_{z}(\text{e}^{i\theta})=\text{Re}\frac{\text{e}^{i\theta}+z}{\text{e}^{i\theta}-z}.$ Notice that $\frac{\partial P_{z}(\text{e}^{i\theta}) }{\partial z}=\frac{\text{e}^{i\theta}}{(\text{e}^{i\theta}-z)^2}$, thus
\begin{align*}
\frac{\partial\hat{\chi}_F}{\partial z}&=\int_{\theta_0}^{\pi-\theta_0}\frac{\text{e}^{i\theta}}{(\text{e}^{i\theta}-z)^2}\frac{d\theta}{2\pi}\nonumber\\
&=\frac{\cos\theta_0}{\pi i}\frac{1}{(\text{e}^{i\theta_0}-z)(\text{e}^{-i\theta}+z)}.
\end{align*}
If $\alpha=\frac{1}{2}+\frac{\theta_0}{\pi}$, such that $L_F(\alpha)$ is the segment joint $-\text{e}^{-i\theta_0}$ and $\text{e}^{i\theta_0}$, we have $$(\text{e}^{i\theta_0}-z)(\text{e}^{-i\theta_0}+z)=1-|z|^2$$
for $z\in L_F(\alpha)$. So
$$\frac{\partial\hat{\chi}_F}{\partial z}=\frac{\cos\theta_0}{\pi i}\frac{1}{1-|z|^2}$$
on $L_F(\alpha)$.

For the general $\alpha\in(0,1)$, we may take $\phi$ to be the automorphism with fixed points $\pm i$ such that $\phi$ maps $-\text{e}^{-i\theta_0}$ and $\text{e}^{i\theta_0}$ to $-\text{e}^{-i\tilde{\theta}_0}$ and $\text{e}^{i\tilde{\theta}_0}$ respectively, where $\tilde{\theta}_0=\pi\alpha-\pi/2$. Let $\tilde{F}=\{\text{e}^{i\theta} : \tilde{\theta}_0<\theta<\pi-\tilde{\theta}_0\}$. Then $\phi(L_F(\alpha))=L_{\tilde{F}}(\alpha)$, which is a segment. So if we write $w=\phi(z)$, then for $z\in L_F(\alpha)$,
\begin{align*}
 \left|\frac{\partial\hat{\chi}_{F}}{\partial z}(z)\right|=&\left|\frac{\partial\hat{\chi}_{\tilde{F}}}{\partial w}(w)\right|\cdot|\varphi'(z)|\\
 =&\frac{\cos\tilde{\theta}_0}{\pi(1-|w|^2)}\cdot\frac{1-|w|^2}{1-|z|^2}\\
 =&\frac{\sin\pi\alpha}{\pi(1-|z|^2)},
\end{align*}
where the first equality follows from the fact that $\hat{\chi}_{F}=\hat{\chi}_{\tilde{F}}\comp\phi$.
\end{proof}

\begin{lemma}\label{lemma2}
Let $F$ be an arc on $\partial D$. Then for any non-negative function $g\in L^1(D)$, we have
$$\int_Dgdv=\int_0^1d\alpha\int_{L_F(\alpha)}\frac{g(z)(1-|z|^2)}{2\sin\pi\alpha}ds(z).$$

\end{lemma}

\begin{proof}
Assume that $F=\{\text{e}^{i\theta} : \theta_0<\theta<\pi-\theta_0\}.$ Let $\tau(z)=\frac{1+\text{e}^{i\theta_0}z}{\text{e}^{i\theta_0}-z}$, then $\tau$ maps $D$ onto the half-plane $\Pi^+=\{z : \text{Re}z>0\}$, and $\tau(L_F(\alpha))=\{r\text{e}^{i\theta} : \theta=\pi\alpha-\pi/2\}$. Therefore,

\begin{align*}
 \int_D g dv=&\frac{1}{\pi}\int_{-\frac{\pi}{2}}^{\frac{\pi}{2}} d\theta\int_0^\infty  g\comp\tau^{-1}(r\text{e}^{i\theta})\cdot\left|(\tau^{-1})'(r\text{e}^{i\theta})\right|^2\cdot rdr\\
 =&\int_0^1 d\alpha\int_{L_F(\alpha)}  g\cdot\omega,
\end{align*}
where
\begin{align*}
 \omega=&\frac{1}{2}\left|\tau'(z)\right|^{-2}d|\tau(z)|^2\\
 =&\frac{|(1+\text{e}^{i\theta}z)(\text{e}^{i\theta}-z)|}{2\cos\theta_0}ds(z)
\end{align*}
on $L_F(\alpha)$. According to the proof of Lemma \ref{lemma1}, for $z\in L_F(\alpha)$ we have
\begin{align*}
\frac{|(1+\text{e}^{i\theta}z)(\text{e}^{i\theta}-z)|}{2\cos\theta_0}=\frac{\pi}{2}\left|\frac{\partial\hat{\chi}_{F}}{\partial z}(z)\right|^{-1}=\frac{(1-|z|^2)}{2\sin\pi\alpha}.
\end{align*}
So
\begin{align*}
\int_{L_F(\alpha)} g\cdot\omega=&\frac{1}{2\sin\pi\alpha}\int_{L_F(\alpha)}g(z)(1-|z|^2)ds(z),
\end{align*}
hence
$$\int_Dgdv=\int_0^1d\alpha\int_{L_F(\alpha)}\frac{g(z)(1-|z|^2)}{2\sin\pi\alpha}ds(z).$$
\end{proof}

The following theorem states that the average of the function $\frac{2N_\varphi(z)}{1-|z|^2}$ on Carleson window $G_F(\alpha)$ equals to the average of the pull-back measures of $G_F(u)$ for $u\in(\alpha,1)$.

\begin{theorem}\label{aver2}
Suppose $\varphi$ an analytic self-map of $D$ and $F$ is an arc on $\partial D$. Then for $\alpha>\hat{\chi}_F(\varphi(0))$ we have
$$\int_{L_F(\alpha)}\frac{N_\varphi(z)}{1-|z|^2}ds(z)=\frac{\pi^2}{\sin\pi\alpha}\int_{\alpha}^1\mu_{\varphi}\left(G_F(u)\right)du.$$
\end{theorem}

\begin{proof}
 Assume that $F=\{\text{e}^{i\theta} : \theta_0<\theta<\pi-\theta_0\}$. Fix $\alpha_1,\alpha_2$ such that $\hat{\chi}_F(\varphi(0))<\alpha_1<\alpha_2<1$, and let

$$I (\alpha) = \left\{ \begin{array}{ll}
     \alpha_1, & 0 \leqslant \alpha < \alpha_1,\\
     \alpha, & \alpha_1 \leqslant \alpha \leqslant \alpha_2,\\
     \alpha_2, & \alpha > \alpha_2.
   \end{array} \right.$$
Then by (\ref{Rudin}),
\begin{align}\label{zuo}
  \int_{\mathfrak{M}_{\infty}} I\comp\hat{\chi}_F d\mu_\varphi= \alpha_1+ \int_{\alpha_1}^{\alpha_2}  \mu_{\varphi}\left(G_F(u)\right)du.
\end{align}
Since $\varphi(0)\notin G_F(\alpha_1),$ then $I\left(|\chi_F\comp\varphi(0)|\right)=\alpha_1$. So we have
\begin{align}\label{you}
\left(|\chi_F\comp\varphi(0)|\right)+\frac{1}{2}\int_D\Delta (I\comp\chi_F)\cdot N_\varphi dv=&\alpha_1+2\int_DI''\comp\chi_F \cdot \left|\frac{\partial\chi_F}{\partial z}\right|^2 N_{\varphi} dv\nonumber\\
=&\alpha_1+\int_0^1 I''(\alpha)\cdot J_\alpha d\alpha\nonumber\\
=&\alpha_1+J_{\alpha_1}-J_{\alpha_2}.
\end{align}
where
\begin{align*}
J_\alpha=&2\int_{L_F(\alpha)}\left|\frac{\partial\chi_F}{\partial z}\right|^2 N_{\varphi}\cdot\frac{1-|z|^2}{2\sin\pi\alpha}ds(z)\\
=&\frac{\sin\pi\alpha}{\pi^2}\int_{L_F(\alpha)}\frac{N_\varphi(z)}{1-|z|^2}ds(z).
\end{align*}
Notice that the last equality of (\ref{you}) relays on that $J_\alpha$ is continuous with respect to $\alpha$. This will be illustrated in the last part of this proof.

By (\ref{Stan}), the equations (\ref{zuo}) and (\ref{you}) should be equal to each other, i.e.,
$$J_{\alpha_1}-J_{\alpha_2}=\int_{\alpha_1}^{\alpha_2}  \mu_{\varphi}\left(G_F(u)\right)du.$$
Note that $\lim_{\alpha_2\to 1}J_{\alpha_2}=0$ since $\frac{N_\varphi(z)}{1-|z|^2}$ is bounded near the boundary. So we have
$$J_{\alpha}=\int_{\alpha}^1\mu_{\varphi}\left(G_F(u)\right)du$$
for $\alpha>\hat{\chi}_F(\varphi(0))$.

Now let's show that $J_\alpha$ is continuous for $\alpha>\hat{\chi}_F(\varphi(0))$. Let $\tau$ be the same map as in the proof of Lemma \ref{lemma2}. Then
$$J_\alpha=\frac{1}{2\pi^2}\int_0^\infty \frac{\overline{N}_\varphi\comp\tau^{-1}(r\text{e}^{i\theta})}{r}dr,$$
where $\theta=\pi\alpha-\pi/2$. Let $\eta(z)=\log z$, which is bi-holomorphic from the half-plane $\Pi^+=\{z : \text{Re}z>0\}$ onto the strip $T=\{z : -\frac{\pi}{2}<\text{Im}z<\frac{\pi}{2}\}$. Then
$$J_\alpha=\frac{1}{2\pi^2}\int_0^\infty\overline{N}_\varphi\comp\tau^{-1}\comp\eta^{-1}(t+i\theta)dt.$$
Then it is clear that $J_\alpha$ is continuous with respect to $\theta$ (as well as $\alpha$), since $\overline{N}_\varphi\comp\tau^{-1}\comp\eta^{-1}$ is subharmonic on $T\backslash\{\eta\comp\tau\comp\varphi(0)\}.$
\end{proof}

As a corollary, we now prove that the average of $\frac{N_\varphi(z)}{-\log|z|}$ on $S(\zeta,h)$ can also be dominated by the pull-back measure. Remark \ref{zh} shows that this corollary is an improvement of Theorem A.

\begin{corollary}\label{zhdkz}
Suppose $\varphi$ is an analytic self-map of $D$. Assume $\zeta\in\partial D$ and $h\in(0,1). $ If $\varphi(0)\notin S(\zeta,h)$, then there exists constant $c_h>1$ depending only on $h$ such that
\begin{align}\label{dom}
\frac{1}{v(S(\zeta,h))}\int_{S(\zeta,h)}\frac{N_\varphi(z)}{-\log|z|}dv(z)\leqslant c_h\cdot\frac{\int_0^h\mu_{\varphi}(\hat{S}(\zeta,t))dt}{\int_0^h\sigma(\hat{S}(\zeta,t))dt}.
\end{align}
Moreover, we can further more require that $\lim_{h\to0}c_h=1$.

\end{corollary}

\begin{proof}
It is sufficient to prove the result to the function $\frac{2N_\varphi(z)}{1-|z|^2}$ instead of $\frac{N_\varphi(z)}{-\log|z|}$.

For $t\in(0,h)$ we have $\hat{S}(\zeta,t)=G_{F_t}(\alpha_t)$ where $F_t$ is the arc $\partial D\cap \overline{S(\zeta,t)}$, and it is easy to see that $\alpha_t\to\frac{1}{2}$ as $t\to0$. So by Theorem \ref{aver2} we have
%It is easy to check that
%$$S(\zeta,h)\subset G_{F_h}(1/2)\subset S(\zeta,c_hh),$$
%where $c_h=1-h/2$ and $F_h$ is the arc $\partial D\cap \overline{S(\zeta,c_hh)}$ . So
\begin{align*}
\int_{S(\zeta,h)}\frac{N_\varphi(z)}{1-|z|^2}dv(z)=&\frac{1}{\pi}\int_0^hdt\int_{L_{F_t}(\alpha_t)}\frac{N_\varphi(z)}{1-|z|^2}ds(z)\\
=&\pi\int_0^hdt\int_{\alpha_t}^1\frac{\mu_{\varphi}\left(G_{F_t}(u)\right)}{\sin\pi\alpha_t}du\\
\leqslant&\pi\int_0^h\frac{1-\alpha_t}{\sin\pi\alpha_t}\cdot\mu_{\varphi}(\hat{S}(\zeta,t))dt\\
\leqslant&\frac{c_h\pi}{2}\int_0^h\mu_{\varphi}(\hat{S}(\zeta,t))dt,
\end{align*}
where $\lim_{h\to0}c_h=1$.
Also note that
$$\lim_{h\to0}\frac{V(S(\zeta,h))}{h^2/2}=\lim_{t\to0}\frac{\sigma(\hat{S}(\zeta,t))}{t/\pi}=1,$$
so by taking the constant $c_h$ slight larger, we have
$$\frac{1}{v(S(\zeta,h))}\int_{S(\zeta,h)}\frac{N_\varphi(z)}{-\log|z|}dv(z)\leqslant c_h\cdot\frac{\int_0^h\mu_{\varphi}(\hat{S}(\zeta,t))dt}{\int_0^h\sigma(\hat{S}(\zeta,t))dt}.$$

\end{proof}

\begin{example}
By taking $\varphi$ to be an inner function with $\varphi(0)=0$, we can see estimation (\ref{dom}) is sharp. In fact, by Lemma 3.27 in \cite{Cowen}, $\frac{N_\varphi(z)}{-\log|z|}$ equals $1$ almost everywhere in $D$ in this situation. So the left side of (\ref{dom}) is $1$ for any $\zeta\in\partial D$ and $h\in(0,1)$ such that $\varphi(0)\notin S(\zeta,h)$. On the other hand, by Theorem 3.8 \cite{Cowen} we can know that $C_\varphi$ is an isometry on $H^2$. Therefore we have $\mu_{\varphi}(\hat{S}(\zeta,t))=\sigma(\hat{S}(\zeta,t)).$ So the right side of (\ref{dom}) equals to $c_h$, which converges to $1$.
\end{example}

\begin{remark}\label{zh}
Note that Corollary \ref{zhdkz} is stronger than the results in \cite{MathA} since Theorem A can be got from this corollary with better constants. In fact, let $a=(1-h/2)\zeta$, then $D(a,h/2)\subset S(\zeta,h)$.
 %Also
%$$\frac{v(D(a,h/2))}{v(S(\zeta,h))}\to\frac{1}{2}$$
%as $h\to0$.
According to the sub-mean value property of $N_\varphi$ and the proof of Corollary \ref{zhdkz} we can know that
\begin{align*}
N_\varphi(a)\leqslant&\frac{1}{v(D(a,h/2))}\int_{D(a,h/2)}N_\varphi(z)dv(z)\\
\leqslant&\frac{8}{h}\int_{S(\zeta,h)}\frac{N_\varphi(z)}{1-|z|^2}dv(z)\\
\leqslant&\frac{4\pi c_h}{h}\int_0^h\mu_\varphi(\hat{S}(\zeta,t))dt\\
\leqslant&4\pi c_h\cdot\mu_\varphi(\hat{S}(\zeta,h)).
\end{align*}
 So
\begin{align*}
\sup_{S(\zeta,h)}N_\varphi\leqslant 4\pi c_h\cdot\mu_\varphi(\hat{S}(\zeta,2h)),
\end{align*}
where $\lim_{h\to0}c_h=1$.

\end{remark}

Finally, we'd like to point out that Theorem \ref{aver2} sheds a light on the behavior of the function $\frac{N_\varphi(z)}{-\log|z|}$ on the Shilov boundary $\mathfrak{M}_{L^{\infty}}$: we strongly suspect that  $\frac{N_\varphi(z)}{-\log|z|}$ is well defined and continuous on $\mathfrak{M}_{L^{\infty}}$. More precisely, we have the following conjecture.

\begin{conjecture}
Let $\tau_\varphi(z)=\frac{\overline{N}_\varphi(z)}{-\log|z|}$. Then the followings hold.

(i) Suppose $m\in\mathfrak{M}_{L^{\infty}}$ and $\{z_\alpha\}$ is a net in $D$ such that $z_\alpha\to m$ in $\mathfrak{M}_{\infty}$, then $\tau(z_\alpha)$ converges. So if we define this limitation the value of $\tau_\varphi$ at $m$, then $\tau_\varphi$ is well defined on $\mathfrak{M}_{L^{\infty}}$.

(ii) Moreover, $\tau_\varphi$ is continuous on $\mathfrak{M}_{L^{\infty}}$, and the restriction of $\tau_\varphi$ on $\mathfrak{M}_{L^{\infty}}$ is exactly the Radon-Nikodym derivative of  $\mu_\varphi|_{\mathfrak{M}_{L^{\infty}}}$ with respect to $\sigma$.
\end{conjecture}

In fact, by Theorem \ref{aver2}, for any Borel set $F\subset\partial D$,

\begin{align*}
\lim_{\alpha\to 1}\int_{L_F(\alpha)}\tau_\varphi(z) ds(z)=&\lim_{\alpha\to 1}\int_{L_F(\alpha)}\frac{2N_\varphi(z)}{1-|z|^2}ds(z)\\
=&\lim_{\alpha\to1}\frac{2\pi^2}{\sin\pi\alpha}\int_{\alpha}^1\mu_{\varphi}\left(G_F(u)\right)du\\
=&2\pi\mu_\varphi\left(\bigcap_{u=\alpha}^1 G_F(\alpha)\right)\\
=&2\pi\mu_\varphi(\hat{F}),
\end{align*}
where the last equality follows from Proposition \ref{pms}. If the statement (\textit{\romannumeral1}) of this conjecture is true, since $\tau_\varphi$ is upper semi-continuous then by a similar way with the proof of Theorem 7 in \cite{jingxiang}, $\tau_\varphi$ has radial limit almost everywhere on $\partial D$. So
\begin{align*}
\lim_{\alpha\to 1}\int_{L_F(\alpha)}\tau_\varphi(z) ds(z)=\int_F\tau_\varphi(\text{e}^{i\theta}) d\theta.
\end{align*}
As a consequence we have
\begin{align}\label{daoshu}
\int_{\hat{F}}\tau_\varphi d\sigma=\mu_\varphi(\hat{F}).
\end{align}
By Theorem 4.3 of Chapter \uppercase\expandafter{\romannumeral5} in \cite{BAF}, $L^\infty(\mathfrak{M}_{L^{\infty}},d\sigma)=C(\mathfrak{M}_{L^{\infty}}).$  So (\ref{daoshu}) implies that $\tau_\varphi$ equals to $d\mu_\varphi/d\sigma$ as functions in $C(\mathfrak{M}_{L^{\infty}})$. Therefore, if one can prove the first part of this conjecture, then the second part follows directly as a corollary of Theorem \ref{aver2}. Note that the second part of this conjecture gives a fabulous explanation why the Nevanlinna counting function and pull-back measure are so closely related with each other.

\vspace*{1cm}
\textbf{Disclosure statement}

No potential conflict of interest was reported by the author(s).

\end{document}